\pdfoutput=1
\documentclass{article}

\usepackage[top=1in, bottom=1.25in, left=1.25in, right=1.25in]{geometry}
\usepackage[pdfusetitle,
            colorlinks=true,
            citecolor=magenta,
            linkcolor=magenta,
            urlcolor=magenta]{hyperref}
\usepackage[style=alphabetic,natbib=true,maxalphanames=7,maxbibnames=99]{biblatex}
\usepackage[utf8]{inputenc}
\usepackage[T1]{fontenc}
\usepackage{latexsym,amsmath,amsfonts,amssymb,stmaryrd,amsthm,mathtools}
\usepackage{tikz-cd}
\usepackage[capitalize,noabbrev]{cleveref}
\usepackage[backgroundcolor=cyan!50!white!50,linecolor=black]{todonotes}
\usepackage{mathpartir}
\usepackage{xparse}
\usepackage{enumitem}
\usepackage{mdframed}
\usepackage{subcaption}
\usepackage{manfnt}
\usepackage{bm}

\DeclareCiteCommand{\thmcite}
  {\usebibmacro{prenote}}
  {\usebibmacro{citeindex}%
   \usebibmacro{cite}}
  {\multicitedelim}
  {\usebibmacro{postnote}}

\mathchardef\mh="2D

\renewcommand{\phi}{\varphi}

\theoremstyle{definition}
\newtheorem{definition}{Definition}[section]

\newtheorem{axioms}{Axiom Set}
\renewcommand{\theaxioms}{\Alph{axioms}}

\theoremstyle{remark}
\newtheorem{remark}[definition]{Remark}

\theoremstyle{plain}
\newtheorem{lemma}[definition]{Lemma}
\newtheorem{theorem}[definition]{Theorem}

\newtheorem{proposition}[definition]{Proposition}
\newtheorem*{proposition*}{Proposition}
\newtheorem{corollary}[definition]{Corollary}
\newtheorem{map}[definition]{Map}

\crefname{axiom}{Axiom}{Axioms}

\newlist{axenum}{enumerate}{2}
\setlist[axenum,1]{label=\theaxioms\arabic*., ref=\theaxioms\arabic*}
\setlist[axenum,2]{label=(\alph*), ref=\theaxioms\theenumi(\alph*)}
\crefalias{axenumi}{axiom}
\crefalias{axenumii}{axiom}

\newcommand{\eqdef}{\coloneqq}

\newcommand{\II}{\mathbb{I}}

\newcommand{\CC}{\ensuremath{\mathcal{C}}}

\newcommand{\GD}{\ensuremath{\Delta}}
\newcommand{\GF}{\ensuremath{\Phi}}
\newcommand{\GG}{\ensuremath{\Gamma}}
\newcommand{\GL}{\ensuremath{\Lambda}}
\newcommand{\GP}{\ensuremath{\Pi}}
\newcommand{\GS}{\ensuremath{\Sigma}}

\newcommand{\id}{\mathrm{id}}
\newcommand{\Id}{\mathrm{Id}}
\newcommand{\dom}{\mathbf{dom}}
\newcommand{\cod}{\mathbf{cod}}
\newcommand{\conc}{\mathbf{conc}}
\newcommand{\diag}[1]{\GD_{#1}}
\newcommand{\pair}[2]{\langle{#1},{#2}\rangle}

\newcommand{\pullback}{\ar[phantom]{dr}[very near start]{\lrcorner}}

\newcommand{\Map}[1]{{#1}\mh\mathrm{\mathbf{map}}}
\newcommand{\DefRet}[1]{\mathrm{\mathbf{DefRet}}(#1)}

\NewDocumentCommand\Cocyl{o m}{\mathbb{I} \mathbin{\oslash\IfValueT{#1}{_{#1}}} {#2}}

\NewDocumentCommand\face{o m m}{\delta^{#2} \mathbin{\oslash\IfValueT{#1}{_{#1}}} {#3}}
\NewDocumentCommand\degen{o m}{\varepsilon \mathbin{\oslash\IfValueT{#1}{_{#1}}} {#2}}
\NewDocumentCommand\bdy{o m}{\partial \mathbin{\oslash\IfValueT{#1}{_{#1}}} {#2}}

\NewDocumentCommand\lface{o m m}{\delta^{#2} \mathbin{\widehat\oslash\IfValueT{#1}{_{#1}}} {#3}}
\NewDocumentCommand\lbdy{o m}{\partial \mathbin{\widehat\oslash\IfValueT{#1}{_{#1}}} {#2}}

\title{\texorpdfstring{Stable factorization \\ from a fibred algebraic weak factorization system}{Stable factorization from a fibred algebraic weak factorization system}}
\author{\texorpdfstring{Evan Cavallo\\Carnegie Mellon University\\\href{mailto:ecavallo@cs.cmu.edu}{\texttt{ecavallo@cs.cmu.edu}}}{Evan Cavallo}}

\date{October 2019}

\begin{document}

\maketitle

\begin{abstract}
  We present a construction of stable diagonal factorizations, used to define categorical models of type
  theory with identity types, from a family of algebraic weak factorization systems on the slices of a
  category. Inspired by a computational interpretation of indexed inductive types in cubical type theory due
  to Cavallo and Harper, it can be read as a refactoring of a construction of van den Berg and Garner, and is
  a new alternative among a variety of approaches to modeling identity types in the literature.
\end{abstract}

The connection between weak factorization systems (wfs's) and identity types, observed notably by Awodey and
Warren \cite{awodey09} and Gambino and Garner \cite{gambino08}, lies at the heart of homotopical
interpretations of type theory. However, as Awodey and Warren recognize, a weak factorization system alone is
not sufficient to interpret Martin-L\"of's intensional identity types in the standard sense
\cite{martin-lof75}, because neither the factorization of maps nor the lifts of left against right maps are
sufficiently structured to model coherence under substitution. Rather, one needs \emph{choices} of
factorizations and lifts \emph{in every context, coherently with respect to reindexing}.

In order to deal with this problem, van den Berg and Garner \cite{van-den-berg12} introduced the notion of a
\emph{cloven wfs} (\emph{cwfs}) on a category, slightly weaker than but similar to the more widely-used
\emph{algebraic wfs} \cite{grandis06,garner07}. A cloven wfs provides the choices of factorizations and lifts
mentioned above in an empty context. However, additional work is required to obtain factorizations and lifts
in an arbitrary context in a coherent way. Van den Berg and Garner thus require an even stronger assumption,
that of a \emph{path object category} \cite[Axioms 1-3]{van-den-berg12}, which induces a cloven wfs but also
satisfies the crucial property that its slices coherently inherit path object category structures. This is
enough to derive a \emph{stable functorial choice of diagonal factorizations} \cite[Definition
3.3.3]{van-den-berg12}, which is more-or-less exactly what is needed to model identity types as part of a
model of Martin-L\"of's intensional type theory.

Unfortunately, the axioms that a path object category $\CC$ must satisfy are somewhat onerous to
check. Standard ``objects of paths'' in typical target categories---such as the exponential $X^{[0,1]}$ of a
topological space $X$, or $A^{\Delta^1}$ of a simplicial set $A$---are \emph{not} path objects in this sense,
because they satisfy only up to homotopy certain laws that are required to hold strictly in a path object
category. (For example, path composition must be strictly associative.) Instead, van den Berg and Garner are
forced to rely on more complex objects, such as Moore paths and their simplicial analogues. Showing that these
objects satisfy the necessary axioms is non-trivial \cite[\S7]{van-den-berg12}.

In this note, we approach the problem from a different angle. Rather than looking for a structure on $\CC$
that induces a cwfs and can be reindexed---thus induces a cwfs on every slice---we instead take such a family
of factorization systems on the slices of $\CC$ as \emph{input}. Specifically, we use Swan's notion of
\emph{fibred algebraic weak factorization system} \cite{swan18b}. We will still need some assumptions to get
from here to a stable functorial choice of diagonal factorizations, namely a functorial cocylinder satisfying
some properties. However, these are fairly permissive; for one, we can use $(-)^{\Delta^1}$ as our cocylinder
in simplicial sets without issue.

We then rely on Swan's general techniques for constructing fibred factorization systems \cite{swan18d}, which
apply in such cases as simplicial and structural cubical sets. Intuitively, the ``Moore-like'' aspect of van
den Berg and Garner's construction is encapsulated in this step, the factorizations in each slice typically
being constructed by an inductive process (i.e., a small object argument). Our contribution, then, is to
observe that we can leverage a technique already being used in the construction of awfs's, rather than
introducing separately the machinery of Moore paths.

We introduce the basic notions in \cref{sec:preliminaries}, give the main construction in
\cref{sec:construction}, and instantiate it with cubical and simplicial sets in \cref{sec:examples}.

\paragraph{Related work}

We have already discussed the connection with the work of van den Berg and Garner. Similar ideas were already
employed by Warren \cite[Chapter 3]{warren08}, as briefly summarized in \cite[\S4.2]{awodey09}. Warren uses
categories with a certain variety of interval (namely, a \emph{cocategory object} satisfying further
conditions) to interpret Martin-L\"of identity types. Like the path objects of van den Berg and Garner, these
must satisfy conditions such as (co)associativity on the nose. Any such interval induces a path object
category structure; thus van den Berg and Garner's results generalize those of Warren
\cite[\S5.4]{van-den-berg12}. Warren's models invalidate the uniqueness of identity proofs but are
1-truncated---any pair of identities between identities are identified. (Warren separately constructs a model
in strict $\omega$-groupoids that is not $n$-truncated for any $n \in \mathbb{N}$ \cite[Chapter 4]{warren08}.)

Swan has also presented two general constructions of a stable functorial choice of diagonal factorizations
\cite{swan18a}. The first takes a \emph{algebraic model structure (ams) with structured weak equivalences} as
input \cite[Theorem 4.13]{swan18a}; while the construction itself is appealingly direct, its hypotheses are
difficult to verify. For this reason, Swan also presents a more specialized construction that requires fewer
assumptions to apply \cite[\S5]{swan18a}, taking a \emph{pre-ams} satisfying axioms based on those used by
Gambino and Sattler \cite{gambino17}. By imposing further conditions, further simplifications are possible
\cites[\S9.1]{cchm}[\S6]{swan18a}[\S2.16]{abcfhl}.

Swan's second construction bears a strong resemblance to our own. Indeed, in many ways the two seem to be
``dual''. Swan uses cofibration-trivial fibration factorization to construct the identity types, where we use
trivial cofibration-fibration factorization; his proof of stability uses pullback-stability of the former
factorization system, while we achieve stability by assuming that the latter is fibred.%
\footnote{%
  We are able to completely avoid discussing cofibrations or trivial fibrations in our own construction,
  restricting attention to a single factorization system. However, we will in fact have a model structure in
  the examples we consider, and our \cref{axiom:awfs:lface} would be more cleanly stated as a condition on the
  other factorization system.
} Both constructions pass through a proof that the reflexivity map $X \to \mathrm{Id}_{\GG}(f)$ is a
deformation retract: Swan uses this to show that the map is a trivial cofibration, where we use it to show
that the map $\mathrm{Id}_{\GG}(f) \to Y \times_\GG Y$ is a fibration. (In \cite{van-den-berg12}, the trivial
cofibrations are \emph{defined} to be the strong deformation retracts.) We leave the teasing out of this
apparent ``duality'' for future work. Because we rely on a fibred factorization system, however, we do exclude
some interesting cases handled by Swan, in particular BCH cubical sets \cite{bch}. (See \cite[\S9.3]{swan18b}
for discussion of this case.)

Other approaches to modeling identity types make use of \emph{regular fibrations} (also called \emph{normal
  fibrations}). The idea is to set up the factorization system so that identity types can be modeled simply by
exponentiation by an interval object. Awodey defines an awfs on cartesian cubical sets such that the
factorization of the diagonal $\Delta : A \to A \times A$ is given by $A \to A^{\mathbb{I}} \to A \times A$
\cite{awodey18}. Gambino and Larrea give a definition of normal fibration in any topos equipped with certain
structure (such as simplicial sets or De Morgan cubical sets). Their definition is at least compatible with
$\Sigma$, $\Pi$, and (of course) identity types; they leave univalent universes and inductive types for future
work \cite{gambino19a}. One weakness of normal fibrations is that they apparently interfere with the
constructive modeling of univalent universes; Swan has demonstrated impossibility results to this effect for a
certain class of models \cite{swan18c}. (On the other hand, Gambino and Henry have made recent progress
towards a constructive simplicial model which interprets identity types by an interval exponential
\cite{gambino19b}.)

Finally, one may rely on general coherence results. Voevodsky's simplicial model of univalent type theory, for
example, uses the exponential $(-)^{\Delta^1}$ \cite[Proposition 2.2.3]{kapulkin12} to obtain weakly stable
identity types (among weakly stable versions of the other type formers) and then applies a coherence theorem,
which constructs a new category from the category of simplicial sets in which the weakly stable structures
become strictly stable. This is effective, but it is a big hammer; approaches such as our own are more
targeted.

This construction was inspired by and adapted from the implementation of identity types (and indexed inductive
types more generally) that the author developed with Robert Harper for a computational cartesian cubical type
theory \cite{cavallo19b}. This note isolates the ideas needed to handle identity types; we expect that
modeling higher inductive types with parameters, even simple ones, will require further assumptions.

\paragraph{Acknowledgments}

First and foremost I am grateful to Robert Harper, without whom this note could not exist. Andrew Swan
provided feedback on multiple drafts and helped me to grasp both his own constructions and others in the
literature. I thank Emily Riehl and Christian Sattler for their useful advice, and Mathieu Anel and Steve
Awodey (among others) for teaching me the basics of model category theory over the years.

I gratefully acknowledge the support of the Air Force Office of Scientific Research through MURI grant
FA9550-15-1-0053. Any opinions, findings and conclusions or recommendations expressed in this material are
those of the author and do not necessarily reflect the views of the AFOSR.

\section{Preliminaries}
\label{sec:preliminaries}

\subsection{Fibred factorization systems}

We briefly recall the concept of algebraic weak factorization system introduced by Grandis and Tholen
\cite{grandis06} under the name \emph{natural weak factorization system}. The distributive law requirement in
\cref{def:awfs} was added later on by Garner \cite[Definition 1]{garner07}; we will not use it explicitly
here, so will neglect to explain what it means.

\begin{definition}
  Write $\conc : \CC^\to \times_\CC \CC^\to \to \CC^\to$ for the functor that takes a composable pair of
  arrows in a category $\CC$ and performs the composition.  A \emph{functorial factorization} on a category
  $\CC$ is a section $\pair{L}{R} : \CC^\to \to \CC^\to \times_\CC \CC^\to$ of $\conc$. We will use the letter
  $K$ to denote the functor $\cod \circ L = \dom \circ R : \CC^\to \to \CC$.
\end{definition}

\begin{definition}
  \label{def:awfs}
  An \emph{algebraic weak factorization system} (\emph{awfs}) on a category $\CC$ consists of a comonad
  $(L,\GF,\GS)$ and monad $(R,\GL,\GP)$ on $\CC^\to$ such that
  \begin{enumerate}
  \item $\pair{L}{R}$ is a functorial factorization,
  \item $\GF_f = (\id,R(f)) : L(f) \to f$ and $\GL_f = (L(f),\id) : f \to R(f)$,
  \item the map $(\GS,\GP) : LR \Rightarrow RL$ is a distributive law.
  \end{enumerate}
\end{definition}

\begin{definition}
  Given an awfs $(L,R)$, we define $\Map{L}$ to be the category of pointed endofunctor coalgebras for
  $(L,\GF)$ and $\Map{R}$ to be the category of pointed endofunctor algebras for $(R,\GL)$.
\end{definition}

Concretely, an object in $\Map{L}$ is a pair $(m,s)$ where $m : A \to B$ and $s : B \to K(m)$ is a diagonal
filler for the square $(L(m),\id) : m \to R(m)$, while an arrow $(m',s') \to (m,s)$ in $\Map{L}$ is a square
$(h,k) : m' \to m$ such that $K(h,k)s' = sk$. Dually, an object in $\Map{R}$ is a pair $(f,p)$ where
$f : X \to Y$ and $p : K(f) \to X$ is a filler for $(\id,R(f)) : L(f) \to f$, and an arrow $(f',p') \to (f,p)$
is a square $(h,k) : f' \to f$ such that $pK(h,k) = hp'$.

\begin{figure}[h!]
  \centering
  \begin{subfigure}[b]{0.3\linewidth}
  \[
    \begin{tikzcd}
      A \ar{d}[left]{m} \ar{r}{L(m)} & K(m) \ar{d}{R(m)} \\
      B \ar{r}[below]{\id} \ar[dashed]{ur}{s} & B
    \end{tikzcd}
  \]
  \caption{an $L$-map structure on $m$}
  \end{subfigure}
  \begin{subfigure}[b]{0.3\linewidth}
  \[
    \begin{tikzcd}
      X \ar{d}[left]{L(f)} \ar{r}{\id} & X \ar{d}{f} \\
      K(f) \ar{r}[below]{R(f)} \ar[dashed]{ur}{p} & Y
    \end{tikzcd}
  \]
  \caption{an $R$-map structure on $f$}
  \end{subfigure}
\end{figure}

Intuitively, the maps $s$ and $p$ shown above provide universal solutions to filling problems against $R$- and
$L$-maps respectively. We will use a few standard facts about algebraic weak factorization systems without
proof: the maps returned by $L$ and $R$ are $L$- and $R$-maps respectively, $L$-maps lift functorially against
$R$-maps, and $R$-maps are functorially closed under composition and pullback. (We refer to, e.g., Riehl
\cite{riehl11} for a proper introduction to awfs's.)

Our goal is to take a category $\CC$ equipped with an awfs and construct a \emph{stable functorial choice of
  diagonal factorizations} \cite[Definition 3.3.3]{van-den-berg12}. Under further conditions on
$\CC$---notably, the Frobenius condition---these can be used to interpret type theory with identity types
\cite[Theorem 3.3.5]{van-den-berg12}. As the further conditions and ultimate construction of a model of type
theory are orthogonal to the construction of the diagonal factorizations, we will not discuss them further
here.

\begin{definition}
  \label{def:diagonal-factorizations}
  A \emph{functorial choice of diagonal factorizations} in an algebraic weak factorization system
  $\pair{L}{R}$ on some $\CC$ is a functor $\Map{R} \to \Map{L} \times_\CC \Map{R}$ that, for every $R$-map
  $f : X \to \GG$, assigns a factorization $(P_f,r_f,p_f)$ of the diagonal $\diag{f} : X \to X \times_\GG X$
  as shown below.
  \[
    \begin{tikzcd}
      { } & P_f \ar[dashed]{dr}{p_f} \\
      X \ar[dashed]{ur}{r_f} \ar{rr}{\diag{f}} & { } & X \times_\GG X
    \end{tikzcd}
  \]
\end{definition}

\begin{definition}
  A functorial choice of diagonal factorizations is \emph{stable} when for every pullback square as shown on
  the left below, the square on the right is also a pullback.
  \[
    \begin{tikzcd}
      X' \pullback \ar{d}[left]{f'} \ar{r}{h} & X \ar{d}{f} \\
      \GG' \ar{r}[below]{k} & \GG
    \end{tikzcd}
    \qquad
    \rightsquigarrow
    \qquad
    \begin{tikzcd}
      P_{f'} \pullback \ar{d}[left]{p_{f'}} \ar{r}{P_{(h,k)}} & P_f \ar{d}{p_f} \\
      \GG' \ar{r}[below]{k} & \GG
    \end{tikzcd}
  \]
\end{definition}

\begin{remark}
  Note that the stability condition above only concerns the underlying \emph{maps} of the factorization, not
  their $L$- or $R$-map structures. As van den Berg and Garner note \cite[Propositions 3.3.6 and
  3.3.7]{van-den-berg12}, this condition is all that is needed.
\end{remark}

To build a choice of diagonal factorizations, our construction will assume that the awfs is one of a
\emph{family} of awfs's defined on each slice of the input category. The following is the specialization of
Swan's notion of fibred awfs over an arbitrary Grothendieck fibration \cite[Definition 4.4.6]{swan18b} to the
particular case of the codomain fibration over $\CC$ \cite[\S7.3]{swan18b}.

\begin{definition}
  A \emph{fibred algebraic weak factorization system} (\emph{fibred awfs}) on a category $\CC$ consists of a
  family of awfs's $\pair{L_\GG}{R_\GG}$ on $\CC/\GG$ for each $\GG \in \CC$ that is stable under reindexing.
\end{definition}

\begin{proposition}
  In a fibred awfs, we have functors $\Map{L_\GG} \to \Map{L_\GD}$ and $\Map{R_\GG} \to \Map{R_\GD}$ for every
  $\sigma : \GD \to \GG$, with the underlying map $(\CC/\GD)^\to \to (\CC/\GG)^\to$ given by reindexing in
  each case.
\end{proposition}

\begin{remark}
  We will use $\rightarrowtail_\GG$ and $\twoheadrightarrow_\GG$ to indicate $L_\GG$-maps and $R_\GG$-maps
  respectively.
\end{remark}

\subsection{Cocylinders and homotopies}

As part of our construction, we will also make use of a \emph{functorial cocylinder} on the target
category. Given an object $X$, the cocylinder $\Cocyl{X}$ is to be thought of as the object of paths in $X$. A
functorial cocylinder is useful primarily because it allows us to say that fillers are unique up to homotopy.

\begin{definition}
  A \emph{functorial cocylinder} in a category $\CC$ consists of a functor $\Cocyl{(-)} : \CC \to \CC$
  equipped with natural transformations $\face0{(-)},\face1{(-)} : \Cocyl{(-)} \to \Id$ and a natural
  transformation $\degen{(-)} : \Id \to \Cocyl{(-)}$ that is a section of both $\face0{(-)}$ and $\face1{(-)}$.
\end{definition}

\begin{definition}
  Given a functorial cocylinder on $\CC$ and an object $\GG \in \CC$, there is an induced functorial
  cocylinder on $\CC/\GG$: given $a : A \to \GG$, we define $\Cocyl[\GG]{(A,a)}$ as the following pullback.
  \[
    \begin{tikzcd}
      \Cocyl[\GG]{(A,a)} \pullback \ar[dashed]{d} \ar[dashed]{r} & \Cocyl{A} \ar{d}{\Cocyl{a}} \\
      \GG \ar{r}[below]{\degen{\GG}} & \Cocyl{\GG}
    \end{tikzcd}
  \]
  It is straightforward to derive the transformations $\face[\GG]{i}{(-)}$ and $\degen[\GG]{(-)}$. We will
  write $\Cocyl[\GG]{A}$ rather than $\Cocyl[\GG]{(A,a)}$ when the arrow can be inferred.
\end{definition}

In the case that that the functorial cocylinder is given by exponentiation with an interval object, the
following definitions specialize to instances of the Leibniz exponential.

\begin{definition}
  Given a functorial cocylinder, we write
  $\bdy{A} \eqdef \pair{\face0{A}}{\face1{A}} : \Cocyl{A} \to A \times A$. For any map $f : X \to Y$, we
  define maps $\lface0{f}$ and $\lbdy{f}$ as follows.
  \[
    \begin{tikzcd}[column sep=2em, row sep=1.5em]
      \Cocyl{X} \ar[bend right=20]{ddr}[below left]{\Cocyl{f}} \ar[dashed]{dr}{\lface0{f}} \ar[bend left=20]{drr}{\face0{X}} \\
      & (\Cocyl{Y}) \times_Y X \ar{d} \pullback \ar{r} & X \ar{d}{f} \\
      & \Cocyl{Y} \ar{r}[below]{\face0{Y}} & Y
    \end{tikzcd}
    \qquad
    \begin{tikzcd}[column sep=2em, row sep=1.4em]
      \Cocyl{X} \ar[bend right=20]{ddr}[below left]{\Cocyl{f}} \ar[dashed]{dr}{\lbdy{f}} \ar[bend left=20]{drr}{\bdy{X}} \\
      & (\Cocyl{Y}) \times_Y (X \times X) \ar{d} \pullback \ar{r} & X \times X \ar{d}{f \times f} \\
      & \Cocyl{Y} \ar{r}[below]{\bdy{Y}} & Y \times Y
    \end{tikzcd}
  \]
\end{definition}

\begin{definition}
  Given maps $f_0,f_1 : X \to Y$, a \emph{homotopy from $f_0$ to $f_1$} is a map $\psi : X \to \Cocyl{Y}$ such that
  $(\face{i}{Y})\psi = f_i$ for $i \in \{0,1\}$. We write $\psi : f_0 \sim f_1$.
\end{definition}

We will need to know that for any awfs $\pair{L}{R}$ satisfying certain conditions, $L$-maps between fibrant
objects (that is, $A$ with an $R$-map structure on $!_A : A \to 1$) give rise to deformation retracts.

\begin{definition}
  A map $f : X \to Y$ is a \emph{(left) deformation retract} if there exists a map $g : Y \to X$ with
  $gf = \id_X$ and a homotopy $\psi : fg \sim \id_Y$. Given a category $\CC$ with a functorial cocylinder, we
  write $\DefRet{\CC}$ for the category whose objects are deformation retracts $(f,g,\psi)$ and morphisms
  $(f',g',\psi') \to (f,g,\psi)$ are commutative squares $(h,k) : f' \to f$ such that $hg' = gk$ and
  $(\Cocyl{k})\psi' = \psi k$.
\end{definition}

\begin{lemma}
  \label{lem:left-to-deformation}
  Fix a category $\CC$ with a functorial cocylinder and an awfs $\pair{L}{R}$, together with a functor
  $\Map{R} \to \Map{R}$ assigning an $R$-map structure to $\lbdy{f}$ for every $R$-map $f$. Then we have a
  functor $\Map{L} \times_{\CC \times \CC} (\Map{R} \times \Map{R}) \to \DefRet{\CC}$ that, for each $L$-map
  $m : A \rightarrowtail B$ with $R$-map structures on $!_A : A \to 1$ and $!_B : B \to 1$, produces a
  deformation retract structure on $m$.
\end{lemma}
\begin{proof}
  We produce $g : B \to A$ and $\psi : mg \sim \id_B$ by solving the following lifting problems.
  \[
    \begin{tikzcd}[row sep=3em, column sep=3em]
      A \ar[tail]{d}[left]{m} \ar{r}{\id_A} & A \ar[two heads]{d}{!_A} \\
      B \ar{r}[below]{!_B} \ar[dashed]{ur}{g} & 1
    \end{tikzcd}
    \qquad\qquad
    \begin{tikzcd}[row sep=3em, column sep=8em]
      A \ar[tail]{d}[left]{m} \ar{r}{(\degen{B})m} & \Cocyl{B} \ar[two heads]{d}{\lbdy{!_B}} \\
      B \ar{r}[below]{\pair{mg}{\id}} \ar[dashed]{ur}{\psi} & B \times B
    \end{tikzcd} \qedhere
  \]
\end{proof}

\section{Constructing a stable factorization}
\label{sec:construction}

\begin{axioms}
  \label{axioms:awfs}
  Fix a category $\CC$ with finite limits and a fibred awfs $\pair{L_\GG}{R_\GG}_{\GG \in \CC}$. We require
  the following.
  \begin{axenum}
  \item A functorial cocylinder $(\Cocyl{(-)},\face0{(-)},\face1{(-)},\degen{(-)})$.
  \item \label{axiom:awfs:lface} A functor $\Map{R_\GG} \to \Map{R_1}$ making $\lface[\GG]0{f}$ an $R_1$-map
    for each $R_\GG$-map $f$.
  \item \label{axiom:awfs:lbdy} A functor $\Map{R_\GG} \to \Map{R_\GG}$ making $\lbdy[\GG]{f}$ an $R_\GG$-map
    for each $R_\GG$-map $f$.
  \end{axenum}
\end{axioms}

\begin{remark}
  For the moment we choose a set of axioms that will take us most directly to our result; in
  \cref{sec:examples}, we present a second set formulated in terms of a (pre) algebraic model structure. In a
  situation where the $\pair{L_\GG}{R_\GG}$ are the trivial cofibrations and fibrations of a fibred model
  structure, \cref{axiom:awfs:lface} can be read intuitively as a combination of more familiar
  conditions. First, that $\lface[\GG]0{(-)}$ takes fibrations to trivial fibrations. Second, that trivial
  fibration structure is fiberwise structure: a map $f \in (\CC/\GG)^\to$ is \emph{globally} a trivial
  fibration as soon as it is a trivial fibration as a map over $\GG$. And finally, that every trivial
  fibration is (of course) a fibration.
\end{remark}
  
We aim to construct a stable functorial choice of diagonal factorizations for the awfs $\pair{L_1}{R_1}$. The
central idea is to define the identity type in context $\GG$ by factorizing the diagonal with respect to
$\pair{L_\GG}{R_\GG}$. This is similar to the approach taken by van den Berg and Garner, the difference being
that we use factorization relative to $\GG$ rather than a path object.
\[
  \begin{tikzcd}[row sep=1.3em]
    { } & K_\GG(\GD_f) \ar[dashed, two heads]{dr}[near start]{R_\GG(\GD_f)}[very near end,above right]{\GG} \\
    X \ar[two heads]{dr}[below left]{f}[very near end,below left]{1} \ar[dashed, tail]{ur}[near end]{L_\GG(\GD_f)}[very near end,below right]{\GG} \ar{rr}{\diag{f}} & { } & X \times_\GG X  \ar[two heads]{dl}{\pair{f}{f}}[very near end,below right]{1} \\
    { } & \GG
  \end{tikzcd}
\]

This definition evidently satisfies the stability condition. However, note that the maps $L_\GG(\GD_f)$ and
$R_\GG(\GD_f)$ are \emph{a priori} only left and right maps with respect to $\pair{L_\GG}{R_\GG}$, whereas
\cref{def:diagonal-factorizations} requires that they be so with respect to $\pair{L_1}{R_1}$. The remainder
of this section is dedicated to establishing that this is indeed the case (and is functorially so).  We will
generally not check functoriality explicitly, as it is straightforward to verify.

The following two lemmas hold in a fibred awfs over any Grothendieck fibration (and with $1$ replaced by any
$\GG'$ in a suitable fashion), but we give concrete proofs for the sake of
self-containedness. \cref{lem:sigma-preserves-left} is the easy half of the result we need: it shows that
$L_\GG(\GD_f)$ above is an $L_1$-map.

\begin{lemma}
  \label{lem:restrict-right}
  We have a functor $(\CC/\GG)^\to \times_{\CC^\to} \Map{R_1} \to \Map{R_\GG}$ that, for each
  $f : (X,x) \to (Y,y)$ with an $R_1$-map structure, provides an $R_\GG$-map structure on $f$.
\end{lemma}
\begin{proof}
  As the awfs is fibred, we have an $R_\GG$-map structure on
  $\GG \times f : (\GG \times X,\pi_1) \to (\GG \times Y,\pi_1)$. We now observe that $f$ is a pullback of
  $\GG \times f$ along $\pair{y}{\id} : Y \to \GG \times Y$, so we can apply the stability of $R_\GG$-maps
  under pullback.
\end{proof}

\begin{lemma}
  \label{lem:sigma-preserves-left}
  We have a functor $\Map{L_\GG} \to \Map{L_1}$ that, for each $m : (A,a) \rightarrowtail_\GG (B,b)$, provides
  an $L_1$-map structure on $m$.
\end{lemma}
\begin{proof} Let $m : (A,a) \rightarrowtail_\GG (B,b)$ be given.  Note that we have
  $\GG \times m : (\GG \times A,\pi_1) \to (\GG \times B, \pi_1)$ and
  $(\pair{a}{\id_A},\pair{b}{\id_B}) : m \to \GG \times m$ in $(\CC/\GG)^\to$.

  We construct the
  necessary filler for $(L_1(m),\id_B) : m \to R_1(m)$ as shown below.
  \[
    \begin{tikzcd}[column sep=4em]
      A \ar[tail]{d}[left]{m}[very near end,left]{\GG} \ar{r}{L_\GG(m)} & K_\GG(m) \ar[two heads]{d}{R_\GG(m)}[very near end,left]{\GG} \ar{rr}{K_\GG(\pair{a}{\id_A},\pair{b}{\id_B})} && K_\GG(\GG \times m) \cong \GG \times K_1(m) \ar[two heads]{d}[very near end,left]{\GG} \ar{r}{\pi_2} & K_1(m) \ar[two heads]{d}{R_1(m)}[very near end,left]{1} \\
      B \ar[dashed]{ur} \ar{r}[below]{\id_B} & B \ar{rr}[below]{\pair{b}{\id}} && \GG \times B \ar{r}[below]{\pi_2} & B
    \end{tikzcd}
  \]
  Here we use the fact that the awfs is fibred to see that the top composite is indeed
  $L_1(m)$.
\end{proof}

The next lemma allows us to transfer an $R_1$-map structure forward along an $L_\GG$-map when the target is at
least an $R_\GG$-map (as a map into the terminal object of $\CC/\GG$). We will use this to derive an $R_1$-map
structure on the map $K_\GG(\GD_f) \to \GG$ from the $R_1$-map structure on $X \to \GG$.

\begin{lemma}
  \label{lem:right-across-left}
  We have a functor
  $\Map{L_\GG} \times_{\CC/\GG \times \CC/\GG} (\Map{R_1} \times_\CC \Map{R_\GG}) \to \Map{R_1}$ that, given
  $m : (A,a) \rightarrowtail_\GG (B,b)$ with an $R_1$-map structure on $a : A \to \GG$ and an $R_\GG$-map
  structure on $b : (B,b) \to (\GG,\id)$, produces an $R_1$-map structure on $b$.
\end{lemma}
\begin{proof}
  Note that $a : (A,a) \to (\GG,\id)$ is also an $R_\GG$-map because the awfs is fibred. By
  \cref{lem:left-to-deformation} (for which we use \cref{axiom:awfs:lbdy}), we see that $m$ is a deformation
  retract in $\CC/\GG$, so we have $r : (B,b) \to (A,a)$ with $rm = \id_A$ and a homotopy
  $\psi : mr \sim \id_B$.

  We need a lift for $(\id,R_1(b)) : L_1(b) \to b$.  We first transform this into a lifting problem against
  $a$ and produce a lift, using the $L_1$-map structure on $L_1(b)$ and $R_1$-map structure on $a$.
  \[
    \begin{tikzcd}[row sep=3em, column sep=3em]
      B \ar[tail]{d}[left]{L_1(b)}[very near end,left]{1} \ar{r}{\id_B} & B \ar[two heads]{d}{b}[very near end,left]{\GG} \ar{r}{r} & A \ar[two heads]{d}{a}[very near end,left]{1} \\
      K_1(b) \ar[dashed, bend left=10]{urr}[near start]{j_a} \ar{r}[below]{R_1(b)} & \GG \ar{r}[below]{\id_\GG} & \GG
    \end{tikzcd}
  \]
  Note that $mj_a : K_1(b) \to B$ makes the lower triangle of the original problem commute, but not the upper
  triangle. We rectify this by solving the following lifting problem in $\CC/\GG$, using
  \cref{axiom:awfs:lface} to obtain an $R_1$-map structure on $\lface[\GG]0{b}$.
  \[
    \begin{tikzcd}[row sep=3em, column sep=8em]
      B \ar[tail]{d}[left]{L_1(b)}[very near end, left]{1} \ar{r}{\psi} & \Cocyl[\GG]{B} \ar[two heads]{d}{\lface[\GG]0{b}}[very near end,left]{1} \\
      K_1(b) \ar[dashed]{ur}{j_b} \ar{r}[below]{\pair{(\Cocyl[\GG]{R_1(b)})\psi}{mj_a}} & (\Cocyl[\GG]{\GG}) \times_\GG B
    \end{tikzcd}
  \]
  The map $(\face1{B})j_b : K_1(b) \to B$ is our desired lift.
\end{proof}

Finally, we show that any $R_\GG$-map between $R_1$-maps is itself an $R_1$-map. Once we have derived an
$R_1$-map structure on the map $K_\GG(\GD_f) \to \GG$ via the previous lemma, we can use this to derive an
$R_1$-map structure on $R_\GG(\GD_f) : K_\GG(\GD_f) \to X \times_\GG X$.

\begin{lemma}
  \label{lem:heterogenize-right}
  We have a functor
  $\Map{R_\GG} \times_{\CC/\GG \times \CC/\GG} (\Map{R_1} \times_\CC \Map{R_1}) \to \Map{R_1}$ that, given
  $f : (X,x) \twoheadrightarrow_\GG (Y,y)$ with $R_1$-map structures on $x$ and $y$, produces an $R_1$-map
  structure on $f$.
\end{lemma}
\begin{proof}
  We must produce a diagonal lift for $(\id,R_1(f)) : L_1(f) \to f$. First, we use the assumption that $x$ is an
  $R_1$-map to solve the following lifting problem.
  \[
    \begin{tikzcd}[row sep=3em, column sep=3em]
      X \ar[tail]{d}[left]{L_1(f)}[very near end,left]{1} \ar{r}{\id_X} & X \ar[two heads]{d}{x}[very near end,left]{1} \\
      K_1(f) \ar[dashed]{ur}{j_x} \ar{r}[below]{yR_1(f)} & \GG
    \end{tikzcd}
  \]
  Next, we use $j_x$ to define a second filling problem, which we solve using \cref{axiom:awfs:lbdy} and the
  assumption that $y$ is an $R_1$-map.
  \[
    \begin{tikzcd}[row sep=3em, column sep=12em]
      X \ar[tail]{d}[left]{L_1(f)}[very near end,left]{1} \ar{r}{(\degen{Y})f} & \Cocyl{Y} \ar[two heads]{d}{\lbdy{y}}[very near end,left]{1} \\
      K_1(f) \ar[dashed]{ur}{j_y} \ar{r}[below]{\pair{(\degen{\GG})yR_1(f)}{\pair{fj_x}{R_1(f)}}} & (\Cocyl{\GG}) \times_{\GG \times \GG} (Y \times Y)
    \end{tikzcd}
  \]
  Finally, we have an $R_1$-map structure on $\lface[\GG]0{f}$ by \cref{axiom:awfs:lface}. Thus we can solve
  the following lifting problem. (Note that in the bottom map, we use the definition of $\Cocyl[\GG]{Y}$ as
  $\GG \times_{\Cocyl{\GG}} (\Cocyl{Y})$.)
  \[
    \begin{tikzcd}[row sep=3em, column sep=8em]
     X \ar[tail]{d}[left]{L_1(f)}[very near end,left]{1} \ar{r}{\degen[\GG]{X}} & \Cocyl[\GG]{X} \ar[two heads]{d}{\lface[\GG]0{f}}[very near end,left]{1} \\
      K_1(f) \ar[dashed]{ur}{j_f} \ar{r}[below]{\pair{\pair{yR_1(f)}{j_y}}{j_x}} & (\Cocyl[\GG]{Y}) \times_Y X
    \end{tikzcd}
  \]
  Our final diagonal filler is then given by $(\face[\GG]1{X})j_f : K_1(f) \to X$.
\end{proof}

Although our ultimate goal is stable factorization of diagonal maps, we can more generally obtain stable
factorization of all maps in $\CC/\GG$ between $R_1$-maps.

\begin{theorem}
  \label{thm:stable-factorization}
  We have a functor
  $\Map{R_\GG} \times_{\CC/\GG \times \CC/\GG} (\Map{R_1} \times_\CC \Map{R_1}) \to \Map{L_1} \times
  \Map{R_1}$ that, given $f : (X,x) \to (Y,y)$ in $\CC/\GG$ with $R_1$-map structures on $x : X \to \GG$ and
  $y : Y \to \GG$, produces $L_1$- and $R_1$-map structures on $L_\GG(f) : X \to K_\GG(f)$ and
  $R_\GG : K_\GG(f) \to Y$ respectively.
\end{theorem}
\begin{proof}
  First, we have an $L_1$-map structure on $L_\GG(f)$ by \cref{lem:sigma-preserves-left}.

  Write $k \eqdef yR_\GG(f) : K_\GG(f) \to \GG$. We have a $R_\GG$-map structure on $y : (Y,y) \to (\GG,\id)$
  because the awfs is fibred. As $R_\GG$-maps are closed under composition, we also have an $R_\GG$-map
  structure on $k : (K_\GG(f),k) \to (\GG,\id)$. We may therefore apply \cref{lem:right-across-left} to derive
  an $R_1$-map structure on $k$.  We now have an $R_\GG$-map
  $R_\GG(f) : (K_\GG(f),k) \twoheadrightarrow_\GG (Y,y)$ between two $R_1$-maps, so we can apply
  \cref{lem:heterogenize-right} to derive an $R_1$-map structure on $R_\GG(f)$.
\end{proof}

\begin{corollary}
  We have a stable functorial choice of diagonal factorizations in $\pair{L_1}{R_1}$.
\end{corollary}
\begin{proof}
  Given an $R_1$-map $f : X \twoheadrightarrow_1 \GG$, we define $P_f \eqdef K_\GG(\diag{f})$, regarding the
  diagonal as a map $\diag{f} : (X,f) \to (X \times_\GG X, f \circ \pi_1)$ in $\CC/\GG$. Note that we have an
  $R_1$-map structure on $f \circ \pi_1 : X \times_\GG X \to \GG$: it is the composite of
  $f : X \twoheadrightarrow_1 Y$ and $\pi_1 : X \times_\GG X \to X$, the latter of which is the pullback of an
  $R_1$-map and therefore an $R_1$-map. Thus $L_\GG(f)$ and $R_\GG(f)$ are $L_1$- and $R_1$-maps by
  \cref{thm:stable-factorization}, so we may take $r_f \eqdef L_\GG(f)$ and $p_f \eqdef R_\GG(f)$.
  Functoriality of this choice follows from functoriality of \cref{thm:stable-factorization}, while stability
  is immediate from the fact that the awfs is fibred.
\end{proof}

\section{Examples}
\label{sec:examples}

We now demonstrate the applicability of our construction. First, we bring our set of axioms closer to the
common examples by considering the situation where the fibred awfs in question constitutes the trivial
cofibrations and fibrations of a model structure. In this case we can recast \cref{axioms:awfs} in a more
natural way. Actually, we will not need a full model structure to do so; we can get by with a \emph{pre-ams} à
la \citet[Definition 2.8]{swan18a}.

\begin{definition}[{\thmcite[Definition 4.4.6]{swan18b}}]
  A fibred awfs $\pair{L_\GG}{R_\GG}_{\GG \in \CC}$ is \emph{strongly fibred} when $\pair{L_\GG}{R_\GG}$
  preserves pullbacks for each $\GG \in \CC$: if $(h,k) : f' \to f$ is a cartesian square, then so is
  $R_\GG(h,k) : R_\GG(f') \to R_\GG(f)$.
\end{definition}

\begin{proposition}[{\thmcite[Proposition 16]{cavallo19c}}]
  \label{prop:strongly-fibred-pullback}
  In a strongly fibred awfs $\pair{L_\GG}{R_\GG}_{\GG \in \CC}$, the $L_\GG$-maps are (functorially) closed
  under pullback for every $\GG \in \CC$.
\end{proposition}

\begin{definition}[{\thmcite[Definition 2.8]{swan18a}}]
  A \emph{pre algebraic model structure (pre-ams)} on a category $\CC$ is a pair of awfs's $\pair{C^t}{F}$ and
  $\pair{C}{F^t}$ on $\CC$ together with a morphism of awfs's $\xi : \pair{C^t}{F} \to \pair{C}{F^t}$.
\end{definition}

We refer to Riehl \cite[Definition 2.14]{riehl11} for a definition of \emph{morphism of awfs's}; we will only
need the following property.

\begin{proposition}
  \label{prop:pre-ams-right}
  A pre-ams $\xi : \pair{C^t}{F} \to \pair{C}{F^t}$ induces a functor $\Map{F^t} \to \Map{F}$.
\end{proposition}

\begin{axioms}
  \label{axioms:ams}
  Fix a category $\CC$ with finite limits and a fibred awfs $\pair{C_\GG^t}{F_\GG}_{\GG \in \CC}$. We
  require the following.
  \begin{axenum}
  \item \label{axiom:ams:cofib} A strongly fibred awfs $\pair{C_\GG}{F_\GG^t}_{\GG \in \CC}$.
  \item \label{axiom:ams:ams} A pre-ams $\xi : \pair{C^t_1}{F_1} \to \pair{C_1}{F^t_1}$.
   
  \item \label{axiom:ams:cocyl} A functorial cocylinder $(\Cocyl{(-)},\face0{(-)},\face1{(-)},\degen{(-)})$.
  \item \label{axiom:ams:lface} A functor $\Map{F_\GG} \to \Map{F^t_\GG}$ making $\lface[\GG]0{f}$ an
    $F^t_\GG$-map for each $F_\GG$-map $f$.
  \item \label{axiom:ams:lbdy} A functor $\Map{F_\GG} \to \Map{F_\GG}$ making $\lbdy[\GG]{f}$ an $F_\GG$-map for each $F_\GG$-map $f$.
  \end{axenum}
\end{axioms}

\begin{lemma}
  Any $\CC$ and $\pair{C^t_\GG}{F_\GG}_{\GG \in \CC}$ satisfying \cref{axioms:ams} also satisfies
  \cref{axioms:awfs}.
\end{lemma}
\begin{proof}
  Only \cref{axiom:awfs:lface} is not immediate. Suppose we are given an $F_\GG$-map
  $f : (X,x) \twoheadrightarrow_\GG (Y,y)$; we need an $F_1$-map structure on $\lface[\GG]0{f}$. Thanks to
  \cref{axiom:ams:ams} and \cref{prop:pre-ams-right}, it suffices to obtain an $F^t_1$-map structure. We
  therefore show that $\lface[\GG]0{f}$ lifts against $C_1$-maps; let some $C_1$-map $m : A \to B$ be given
  with a lifting problem $(h,k) : m \to \lface[\GG]0{f}$ in $\CC$. By composing the legs of $\lface[\GG]0{f}$
  with $h$ and $k$, we obtain $a,b$ such that $m$ is a map $(A,a) \to (B,b)$ over $\GG$. By reindexing $m$
  along $!_\GG : \GG \to 1$, we obtain an $C_\GG$-map structure on
  $\GG \times m : (\GG \times A, \pi_1) \to (\GG \times B, \pi_1)$. By \cref{axiom:ams:cofib} and
  \cref{prop:strongly-fibred-pullback}, we can pull back $\GG \times m$ back along
  $\pair{b}{\id} : (B,b) \to (\GG \times B,\pi_1)$ to obtain an $C_\GG$-map structure on
  $m : (A,a) \to (B,b)$. Thus $m$ lifts against $\lface[\GG]0{f}$ in $\CC/\GG$ by \cref{axiom:ams:lface},
  which in particular gives a filler for $(h,k)$.
\end{proof}

We now shift to a relatively concrete situation, applying the framework developed in
\cite{cavallo19c,cavallo20} to obtain stable functorial diagonal factorizations in a variety of cubical
models.

\begin{axioms}
  \label{axioms:cms}
  Fix a locally cartesian closed category $\CC$ with finite colimits and disjoint coproducts. We require the
  following.
  \begin{axenum}
  \item A monomorphism $\top : \Phi_{\mathsf{true}} \to \Phi$ such that pullbacks of $\top$ include $!_0 : 0 \to 1$ and $!_1 : 1 \to 1$ and are closed under binary union.
  \item An interval object $\delta_0, \delta_1 : 1 \to \II$ such that $\delta_0$ and $\delta_1$ are pullbacks
    of $\top$.
  \item One of the following two conditions:
    \begin{axenum}
    \item \label{axiom:cms:decidable} $\CC$ is an (internal) category of presheaves and $\top$ is a locally decidable monomorphism.
    \item \label{axiom:cms:wisc} $\CC$ is a $\Pi W$-pretopos and it satisfies the axiom weakly initial set of covers (WISC).
    \end{axenum}
  \end{axenum}
\end{axioms}

\begin{remark}
  An interval $\II$ induces interval objects $\II_\GG \eqdef \pi_2 : \II \times \GG \to \GG$ in each slice
  category. For any map $f : A \to B$, we define the \emph{mapping cylinder} $\mathrm{Cyl}(f)$ to be the
  pushout of $B \overset{f}\longleftarrow A \overset{\delta_0 \times A}{\longrightarrow} \II \times A$; we
  have a map
  $d(f) : A \overset{\delta_1 \times A}{\longrightarrow} \II \times A \overset{\iota_2}{\longrightarrow}
  \mathrm{Cyl}(f)$.
\end{remark}

\begin{map}[{\thmcite[\S1.4]{cavallo19c}}]
  \label{map:cms-generator}
  Let $\CC$ satisfying \cref{axioms:cms} be given.  Write $\Delta$ for the map
  $1_{\Phi \times \II} \to \II_{\Phi \times \II}$ in $\CC/\Phi \times \II$ given by
  $\pair{\pi_2}{\id} : \Phi \times \II \to \II \times (\Phi \times \II)$.  Write $\mathbb{T}$ for the object
  of $\CC/\Phi \times \II$ given by $\top \times \II : \Phi_{\mathsf{true}} \times \II \to \Phi \times
  \II$. Write $T$ for the map $\mathbb{T} \to 1_{\Phi \times \II}$. Define the following map in
  $\CC/\Phi \times \II$, where $+$, $\times$, $\hat\times$, and $\mathrm{Cyl}$ are all computed in the slice
  category.
  \[
    \begin{tikzcd}
      1 +_{\mathbb{T}} (\mathrm{Cyl}(\Delta) \times \mathbb{T})
      \ar[rr, "d(\Delta) \hat\times T"]
      \ar[dr]
      & & \mathrm{Cyl}(\Delta) \ar[dl] \\
      & \Phi \times \II &
    \end{tikzcd}
  \]
\end{map}

\begin{theorem}
  \label{thm:cms-id}
  Assuming \cref{axioms:cms}, the fibred awfs generated (in the sense of \cite{swan18b}) by
  \cref{map:cms-generator} exists and supports a stable functorial choice of diagonal factorizations.
\end{theorem}
\begin{proof}
  Per \cite[Theorem 13]{cavallo19c}, the existence of the fibred awfs follows from \cite[Theorem
  6.14]{swan18d} if \cref{axiom:cms:decidable} holds and \cite[Corollary 6.12]{swan18d} if
  \cref{axiom:cms:wisc} holds. For the same reasons, we have a second fibred awfs generated by $\top$ viewed
  as a map into the terminal object of $\CC/\Phi$; by \cite[Corollary 7.5.5]{swan18b}, this awfs is strongly
  fibred.  Constructing the comparison map is a matter of functorializing the proof in \cite[Theorem
  35]{cavallo19c} that $C^t \subseteq C$ as classes of maps; we leave this to the reader, referring to
  \cite[\S5.4.1]{swan18a} for an example of such an argument. Finally, \cref{axiom:ams:lface,axiom:ams:lbdy}
  are fulfilled by \cite[Lemma 32]{cavallo19c}.
\end{proof}

As detailed in \cite{cavallo20}, \cref{thm:cms-id} can be instantiated to give models of identity types in De
Morgan cubical sets \cite{cchm} and cartesian cubical sets \cite{angiuli18,abcfhl} among other
variations. (Substructural cubical sets as developed by \citet{bch} are, however, \emph{not} an instance of
this construction.)

We can also apply \cref{thm:cms-id} to simplicial sets. If we take $\top$ to be the map
$\mathsf{true} : 1 \to \Omega$ into the subobject classifier, then the class of fibrations generated by
\cref{map:cms-generator} is the same as that of the classical model structure on simplicial sets: it coincides
with the definition of fibration used in \cite{cchm,orton18} by \cite[\S2.3.2]{cavallo20}, and this agrees
with the classical definition by \cite[Chapter IV, \S2]{gabriel67} (as observed in \cite[Corollary
8.4]{sattler17}).

\printbibliography

\end{document}